\documentclass[12pt,reqno]{amsart}
\usepackage{amsmath}
\usepackage{amssymb}
\usepackage[left=2.8cm,top=2.8cm,right=2.8cm,bottom=2.8cm]{geometry}
\usepackage{epsfig,color}
\usepackage{array}
\usepackage[bookmarks]{hyperref}

\begin{document}
\newtheorem{theorem}{Theorem}
\newtheorem{lemma}[theorem]{Lemma}
\newtheorem{definition}[theorem]{Definition}
\newtheorem{conjecture}[theorem]{Conjecture}
\newtheorem{proposition}[theorem]{Proposition}
\newtheorem{algorithm}[theorem]{Algorithm}
\newtheorem{corollary}[theorem]{Corollary}
\newtheorem{question}{Question}
\newtheorem{observation}[theorem]{Observation}
\newtheorem{problem}[theorem]{Open Problem}
\newcommand{\noin}{\noindent}
\newcommand{\ind}{\indent}
\newcommand{\al}{\alpha}
\newcommand{\om}{\omega}
\newcommand{\pp}{\mathcal P}
\newcommand{\ppp}{\mathfrak P}
\newcommand{\R}{{\mathbb R}}
\newcommand{\N}{{\mathbb N}}
\newcommand{\Z}{{\mathbb Z}}
\newcommand\eps{\varepsilon}
\newcommand{\E}{\mathbb E}
\newcommand{\Prob}{\mathbb{P}}
\newcommand{\pl}{\textrm{C}}
\newcommand{\dang}{\textrm{dang}}
\renewcommand{\labelenumi}{(\roman{enumi})}
\newcommand{\bc}{\bar c}
\newcommand{\A}{\mathcal A}
\newcommand{\B}{\mathcal B}
\newcommand{\K}{\mathcal K}
\newcommand{\sn}{\mathrm{sn} }
\newcommand{\Av}{\mathrm{Av} }
\newcommand{\G}{\mathcal G}
\newcommand{\GHex}{G_{hex}}
\newcommand{\GTri}{G_{\!\bigtriangleup}}
\newcommand{\GSq}{G_{\square}}
\newcommand{\remove}[1]{}

\title{Firefighting on square, hexagonal, and triangular grids}

\author{Tom\'{a}\v{s} Gaven\v{c}iak}
\address{Department of Applied Mathematics, Charles University, Praha, Czech Republic}
\email{\texttt{gavento@kam.mff.cuni.cz}}
\thanks{The first author is supported by the grant SVV-2013-260103.}

\author{Jan Kratochv\'{i}l}
\address{Department of Applied Mathematics, Charles University, Praha, Czech Republic}
\email{\texttt{honza@kam.mff.cuni.cz}}
\thanks{The second author is supported by Czech Research grant CE-ITI GA\v{C}R P202/12/6061.}

\author{Pawe\l{} Pra\l{}at}
\address{Department of Mathematics, Ryerson University, Toronto, ON, Canada}
\email{\texttt{pralat@ryerson.ca}}
\thanks{The third author gratefully acknowledges support from NSERC and Ryerson University}


\begin{abstract} 
In this paper, we consider the \emph{firefighter problem} on a graph $G=(V,E)$ that is either finite or infinite. Suppose that a fire breaks out at a given vertex $v \in V$. In each subsequent time unit, a firefighter protects one vertex which is not yet on fire, and then the fire spreads to all unprotected neighbors of the vertices on fire. The objective of the firefighter is to save as many vertices as possible (if $G$ is finite) or to stop the fire from spreading (for an infinite case).

The surviving rate $\rho(G)$ of a finite graph $G$ is defined as the expected percentage of vertices that can be saved when a fire breaks out at a vertex of $G$ that is selected uniformly random. For a finite square grid $P_n \square P_n$, we show that $5/8 + o(1) \le \rho(P_n \square P_n) \le 67243/105300 + o(1)$ (leaving the gap smaller than 0.0136) and conjecture that the surviving rate is asymptotic to 5/8.

We define the surviving rate for infinite graphs and prove it to be $1/4$ for the infinite square grid, even for more than one (but finitely many) initial fires. For the infinite hexagonal grid we provide a winning strategy if two additional vertices can be protected at any point of the process, and we conjecture that the firefighter has no strategy to stop the fire without additional help. We also show how the speed of the spreading fire can  be reduced by a constant multiplicative factor. For triangular grid, we show that two firefighters can slow down the fire in the same sense, which is relevant to the conjecture that two firefighters cannot contain the fire on the triangular grid, and also corrects a previous result of Fogarty~\cite{F}.
\end{abstract}

\maketitle


\section{Introduction}\label{sec:intro} 

The following \emph{firefighter problem} on a graph $G=(V,E)$ was introduced by Hartnell at a conference in 1995~\cite{Hartnel}. Suppose that a fire breaks out at a given vertex $v \in V$. In each subsequent time unit (called a turn), a firefighter \emph{protects} one vertex which is not yet on fire and then the fire spreads to all unprotected neighbors of the vertices already on fire. Once a vertex is on fire or is protected it stays in such state forever. Protecting a vertex is in essence equivalent to deleting it from the graph.

The game stops if no neighbor of the vertices on fire is unprotected and the fire cannot spread. If the graph is finite, the game finishes at some point and the goal of the firefighter is to save as many vertices as possible. In case of an infinite graph, the goal of the firefighter is to stop the fire from spreading or, if this is not possible, to save as many vertices as possible in the limit (we introduce this graph parameter in Section~\ref{sec:inf_graphs}).

Today, almost 20 years later, our knowledge about this problem is much greater and a number of papers have been published. We would like to refer the reader to the survey of Finbow and MacGillivray for more information~\cite{FM}.

\medskip
For finite graphs, we focus on the following property. Let $\sn(G,v)$ denote the number of vertices in $G$ the firefighter can save when the fire breaks out at a vertex $v \in V$, assuming the best strategy is used. Then let $\rho(G,v)=\sn(G,v)/n$ be the proportion of vertices saved (here and throughout the paper, $n$ denotes the number of vertices of $G$, assuming $G$ is finite).
The \emph{surviving rate} $\rho(G)$ of $G$, introduced in~\cite{FHLS}, is defined as the expected $\rho(G,v)$ when the fire breaks out at a random vertex $v$ of $G$ (uniform distribution is used), that is, 
$$
\rho(G) = \frac{1}{n}\sum_{v\in V}\rho(G,v) = \frac{1}{n^2}\sum_{v \in V}\sn(G,v).
$$

For example, it is not difficult to see that for cliques $\rho(K_n) = \frac{1}{n}$, since no matter where the fire breaks out only one vertex can be saved. For paths we get that
$$
\rho(P_n) = \frac {1}{n^2} \sum_{v \in V} \sn(G,v) = \frac {1}{n^2} \left( 2(n-1) + (n-2)(n-2) \right) = 1-\frac {2}{n} + \frac {2}{n^2}
$$
(one can save all but one vertex when the fire breaks out at one of the leaves; otherwise two vertices are burned).

\medskip
It is not surprising that almost all vertices on a path can be saved, and in fact, all trees have this property. Cai, Cheng, Verbin, and Zhou~\cite{CCVZ} proved that  the greedy strategy of Hartnell and Li~\cite{HL} for trees saves at least $1 - \Theta(\log n / n)$ percentage of vertices on average for an $n$-vertex tree. Moreover, they managed to prove that for every outer-planar graph $G$, $\rho(G) \ge 1 - \Theta(\log n / n)$. Both results are asymptotically tight and improved upon earlier results of Cai and Wang~\cite{CW}. (Note that there is no hope for a similar result for planar graphs, since, for example, $\rho(K_{2,n}) = 2/(n+2) = o(1)$.) However, this does not mean that it is easy to find the exact value of $\rho(G)$. It is known that the decision version of the firefighter problem is NP-complete even for trees of maximum degree three~\cite{FKGR}. 

Moving to another interesting direction, the third author of this paper showed that any graph $G$ with average degree strictly smaller than $30/11$ has the surviving rate bounded away from zero~\cite{P} and showed that this result is sharp (the construction uses a mixture of deterministic and random graphs). (See~\cite{P2} for a generalization of this result for the $k$-many firefighter problem.) These results improved earlier observations of Finbow, Wang, and Wang~\cite{FWW}.

\subsection{Our contribution}\label{sec:results} 

First, we study the surviving rate of $P_n \square P_n$, the Cartesian product of two paths of length $n-1$. It was announced by Cai and Wang that
$$
0.625 +o(1) = \frac {5}{8} + o(1) \le \rho(P_n \square P_n) \le \frac {37}{48} + o(1) \approx 0.7708
$$
but a formal proof has not been published. We will prove the following result, which provides much better upper bound.

\begin{theorem}\label{thm:finite}
For the Cartesian product of two paths we have
$$
0.625 + o(1) = \frac {5}{8} + o(1) \le \rho(P_n \square P_n) \le \frac {67243}{105300} + o(1) < 0.6386.
$$
\end{theorem}

Our proof for the upper bound is not very sophisticated and there are ways to improve it. On the other hand, it narrows down the surviving rate to a small interval smaller than 0.0136. It is natural to conjecture the following but this still remains open.
\begin{conjecture}
$\lim_{n \to \infty}  \rho(P_n \square P_n) = 5/8$.
\end{conjecture}

\bigskip
For an infinite graph $G=(V,E)$, the primary goal is to determine if the fire can be stopped from spreading or not. All graphs we discuss here are vertex transitive so the choice of the starting point is irrelevant.

It is known (and easy to show) that it is impossible to surround the fire with one firefighter in the infinite Cartesian grid (see~\cite{WM, F}). On the other hand, it is clear that two firefighters can stop the fire (that is when two vertices can be protected in each round) and in~\cite{WM} the optimal strategy was provided that does it in 8 steps. (See~\cite{32fraction} for a fractional version of this problem.) It was proved in~\cite{F} that if the fire breaks out on the triangular grid, three firefighters contain the fire easily, but the proof that two firefighters
can not contain the fire is unfortunately flawed (as we discuss below) and this question is still open.

\medskip
{\bf For the infinite square grid} $\GSq$, we show that it is optimal to save a $90^\circ$ wedge of vertices. In Section~\ref{sec:inf_graphs} we formally introduce a measure of the surviving rate for infinite graphs and show that $\rho(\GSq)=1/4$.

\medskip
{\bf For the infinite hexagonal grid} $\GHex$, we show that one firefighter can save 2/3 of the grid,
and with just a little additional help of two extra protected vertices, it is possible to stop the fire from spreading:

\begin{theorem}\label{thm:hex}
    $\rho(\GHex)\geq 2/3$.

    Moreover, if the firefighter is allowed to protect one extra vertex at time $t_1$ and one at time $t_2$,
    $1 \leq t_1 \leq t_2$ (possibly with $t_1=t_2$), the firefighter will contain the fire on $\GHex$.
    Moreover, the strategy does not need to know $t_1$ and $t_2$ in advance.
\end{theorem}

\medskip
We also show a strategy to slow down the fire by a constant factor in the following sense.
Here and throughout the paper $N_t(v)=\{x \in  V(G):d_G(x,v)=t\}$ denotes the $t$-th neighborhood of $v$ in the graph $G$,
and $N_{\leq t}(v)=\{x \in V(G) :d_G(x,v)\leq t\} = \bigcup_{s \le t} N_s(v)$.

\begin{theorem}\label{thm:hex_slowdown}
    There exists a universal constant $c<1$ such that when only finitely many vertices of $\GHex$ are burning,
    there exist a vertex $v_0\in V$ and a strategy such that for every large enough $T$ all vertices burning after turn
    $T$ are contained in $N_{\leq c\cdot T}(v_0)$.
\end{theorem}

Finally, even though one can contain the fire with just 2 extra protected vertices and can
slow down the fire, we conjecture that one firefighter per turn alone still cannot stop the fire from spreading.

\begin{conjecture}\label{con:hex_not_saved}
    If a fire breaks out on the hexagonal grid, one firefighter does not suffice to contain the fire.
\end{conjecture}

\medskip
{\bf For the infinite triangular grid} $\GTri$ we show that two firefighters can
slow down the fire to keep $N_{T}(v_0)$ non-burning after $T$ turns, and even slow it down by a constant factor
as in Theorem~\ref{thm:hex_slowdown}:

\begin{theorem}\label{thm:triangle_slowdown}
    There exists a universal constant $c<1$ such that when only finitely many vertices of $\GTri$ are burning,
    there exist a vertex $v_0\in V$ and a strategy for two firefighters such that for every large enough $T$ all vertices burning after turn
    $T$ are contained in $N_{\leq c\cdot T}(v_0)$.
\end{theorem}

This contradicts a statement in~\cite{F} in the proof of Theorem 21 stating that two firefighters can not contain the fire on
the triangular grid. The proof has been noted to be flawed by one of the reviewers and perhaps others.
Indeed, in the proof of Theorem 21 in~\cite[p.~34]{F} it is incorrectly stated that $|N^+(A)|\geq|A|+2$
($N^+(A)=N(A)\cap N_{k+1}(v_0)$ in our notation), but that does not hold for every
$A\subseteq B_k$ (burning part of $N_k(v_0)$ in their paper) as required by Theorem~1 of~\cite{F};
for example with $A=\{v\}$ with $v\in B_k$ not on cone boundary we have $|N^+(A)|=2$.

Our result shows that it is not possible to prove that two firefighters are not sufficient argumenting
only with the burning vertices of $N_t(v_0)$ in turn $t$, as there might be none,
but it would still seem possible (or even likely) that two firefighters can not contain the fire on the triangular grid.
As far as we know, this has been conjectured but we are not aware of any written reference.

\section{Finite square grid}\label{sec:square_grid} 

A \emph{square grid graph} $P_n \square P_n = (V,E)$ is the graph whose vertices correspond to the points in the plane with integer coordinates from
$$
C = \{ -\lfloor n/2 \rfloor, -\lfloor n/2 \rfloor+1,  \ldots, -1, 0, 1, \ldots,  \lceil n/2 \rceil -1 \}
$$
and two vertices are connected by an edge whenever the corresponding points are at distance 1. In other words,
\begin{eqnarray*}
V &=& \{ (a,b) : a,b \in C \}, \\
E &=& \{ vu : v,u \in V \text{ and } \| v-u \| =1 \}.
\end{eqnarray*}
Vertices $(a, \lceil n/2 \rceil -1), a \in C$ form the \emph{north border}. Similarly, vertices $(a, -\lfloor n/2 \rfloor), a \in C$ form the \emph{south border}, $(-\lfloor n/2 \rfloor, b), b \in C$ form the \emph{west border}, and vertices $(\lceil n/2 \rceil -1, b), b \in C$ form the \emph{east border}.

\medskip
We prove the lower bound and the upper bound stated in Theorem~\ref{thm:finite} in two separate subsections.

\subsection{Lower bound}\label{sec:lower_bound} 

Consider the square grid $P_n \square P_n$ for some integer $n$. Suppose that a fire breaks out at a vertex $(a,b)$. Due to the symmetry, we may assume that $0 \le a \le  \lceil n/2 \rceil -1$ and that $0 \le b \le a$. The firefighter can protect the following sequence of vertices in the first few rounds (see Figure~\ref{fig:defense}~(a)):
$$
(a-1,b), (a-1,b+1), (a-2,b-1),(a-2,b+2), \ldots
$$
Once the north border is reached (note that the last vertex protected is $(a-\lceil n/2 \rceil + 1 + b, \lceil n/2 \rceil -1)$), the firefighter goes straight down to the south border protecting the sequence
$$
(a-\lceil n/2 \rceil +b, 2b-\lceil n/2 \rceil +1), (a-\lceil n/2 \rceil +b, 2b-\lceil n/2 \rceil), \ldots.
$$
(For a `big picture' of this strategy see Figure~\ref{fig:defense}~(b).)

\begin{figure}[htbp]
\begin{center}
\begin{tabular}{ccc}
    \vtop{\null\hbox{\includegraphics{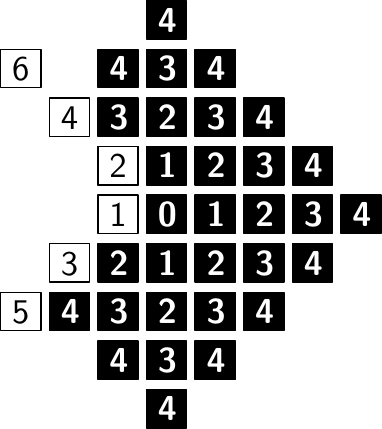}}}
    & \hspace{2cm} &
    \vtop{\null\hbox{\includegraphics{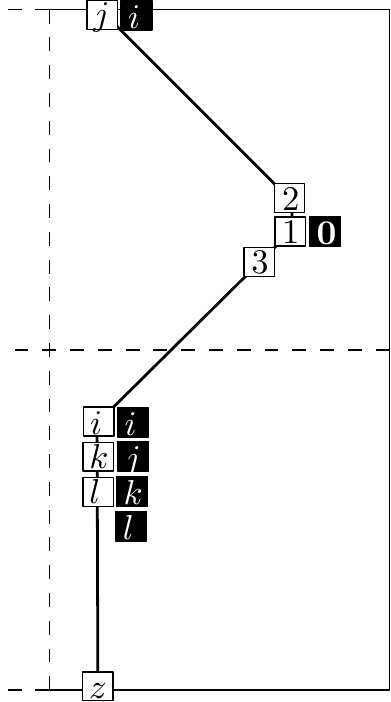}}} \\
    (a) & & (b)
\end{tabular}
\end{center}
    \caption{The beginning of the defense strategy and the `big picture'. The numbers indicate the turn in which the vertex was lit or protected,
    $i$ indicates the turn before firefighters reach north border and $j,k,\dots z$ indicate later turns.} \label{fig:defense}
\end{figure}

Suppose that $a=xn+o(1)$ and $b=yn+o(1)$ for some $0 \le x \le 1/2$ and $0\le y \le x$. Using the described strategy, it is easy to see what percentage of vertices can be saved, which gives us the following bound.
$$
\frac {\sn(P_n \square P_n, (a,b) )}{n^2} \ge \left( \frac 12 -y \right)^2 + \left( \frac 12 +x \right) - \left( \frac 12 -y \right) +o(1) = \left( \frac 12 -y \right)^2 + x + y + o(1).
$$
Hence, since there are 8 symmetric regions to consider,
\begin{eqnarray*}
\rho(P_n \square P_n) &=& \frac {1}{n^2} \sum_{(a,b) \in V} \frac {\sn(P_n \square P_n, (a,b) )}{n^2} \\
& \ge & 8 \int_{0}^{1/2} \int_{0}^{x} \left( \left( \frac 12 -y \right)^2 + x + y \right) dy dx + o(1) \\
& = & 8 \int_{0}^{1/2} \left( - \frac 13 \left( \frac 12 -x \right)^3 + \frac 32 x^2 + \frac {1}{24} \right) dx + o(1) \\
&=& \frac {5}{8} + o(1).
\end{eqnarray*}
The lower bound holds.

\subsection{Upper bound}\label{sec:upper_bound} 

Consider the square grid $P_n \square P_n$ for some integer $n$. For a given vertex $(a,b)$ and $r \in \N$, we abuse slightly the notation and use $N_r(a,b)$ instead of $N_r((a,b))$, the set of vertices at distance $r$ from $(a,b)$. We will use $N^{NE}_r(a,b)$ to denote vertices of $N_r(a,b)$ of the form $(a+s, b+r-s), s = 0, 1, \ldots, r$ and call such vertices \emph{North-East fire-front}. Fire-fronts to other directions and sets $N^{NW}_r(a,b)$, $N^{SE}_r(a,b)$, $N^{SW}_r(a,b)$ are defined analogously. Finally, note that the intersection of any two fire-fronts may be non-empty (for example, $N^{NE}_r(a,b) \cap N^{NW}_r(a,b) = \{ (a,b+r) \}$, provided that $b+r \in C$). We will call such vertices \emph{corners}.

We start with the following simple but very powerful observation.

\begin{lemma}\label{lem:dist_r}
Suppose that a fire breaks out at a vertex $(a,b)$ of $P_n \square P_n$. Regardless of the strategy used by the firefighter, there are at most $r$ vertices in $N_r(a,b)$ that are not burning at time $r$, for every $r \ge 1$.
\end{lemma}
\begin{proof}
In order to prove the theorem, we prove the following stronger claim:
\emph{At time $r \ge 1$, for every non-burning vertex $v$ of $N_r(a,b)$ there is a path $P_v$ from $v$ to some protected vertex $p(v)$. It is allowed that a path is trivial (that is, $v=p(v)$) when $v$ is itself protected. Moreover, all the paths $P_v$ are vertex disjoint.}

\begin{figure}[htbp]
\begin{center}
      \includegraphics{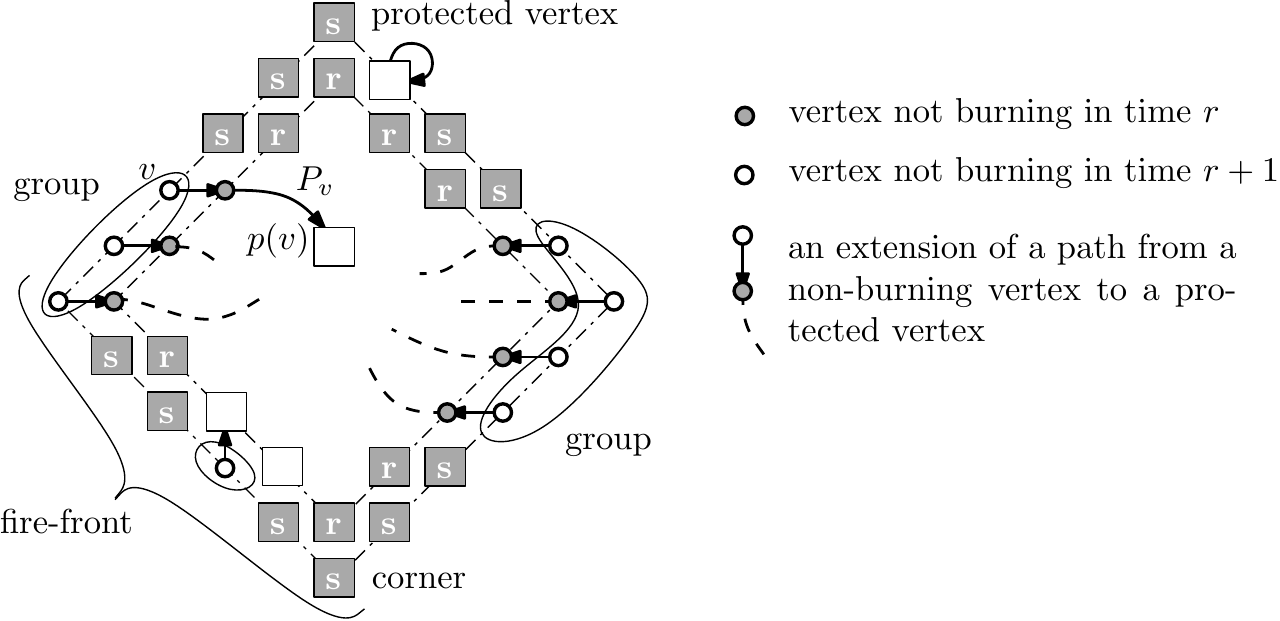}
\end{center}
    \caption{Extending paths in time $r$ to time $s=r+1$.} \label{fig:k-th-move}
\end{figure}

We prove the claim by induction. Clearly, the property holds for $r = 1$: if a neighbor of $(a,b)$ is not burning at time $r=1$, then it must be protected. Suppose that the property holds for $r \ge 1$, our goal is to show that it holds for $r+1$. Let $v$ be a vertex of $N_{r+1}(a,b)$ that is not burning at time $r+1$. If $v$ is protected, then it yields a trivial path. Suppose then that $v$ is not protected. It is clear that no neighbor of $v$ was burning in time $r$; otherwise, $v$ would be on fire in time $r+1$ too.
Hence, each neighbor of $v$ in $N_r(a,b)$ is associated with a unique path; $v$ has two such neighbors, unless $v$ is a corner vertex in which case there is only one such neighbor.
It follows that paths can be extended to all non-burning vertices of $N_{r+1}(a,b)$ by choosing one extension direction for every
\emph{group} of non-burning vertices (see Figure~\ref{fig:k-th-move}). This is always possible \emph{unless} the whole fire-front (one of NE, NW, SE, or SW) is a non-burning group that does not touch the boarder (that is, both corners and their neigbours are in C). Fortunately, this situation cannot occur since this would imply that there were $r+1$ vertices of $N_r(a,b)$ not burning at time $r$ and therefore $r+1$ paths to $r+1$ protected vertices, which is impossible. (Note that $|N_r(a,b)|=4r$ but each fire-front consists of $r+1$ vertices, including two corner vertices.) Therefore the claim holds for $r+1$ and the proof is finished.
\end{proof}

Before we move to investigating the surviving rate of $P_n \square P_n$ let us focus on the case $(a,b)=(0,0)$ in order to explain the idea in a simple setting. Consider first the graph $G$ induced by the set of vertices at distance at most $\lceil n/2 \rceil - 1$ from $(0,0)$ (that is, `diamond shape' square grid). It follows from Lemma~\ref{lem:dist_r} that for every $1 \le r \le \lceil n/2 \rceil - 1$, the fraction of vertices from $N_r(0,0)$ that are saved is at most 1/4 (since $|N_r(0,0)|=4r$ for $1 \le r \le \lceil n/2 \rceil - 1$). On the other hand, the strategy for the firefighter provided in the previous subsection guarantees that this can be achieved. Hence, $\sn(G, (0,0))/|V(G)| = 1/4 + o(1).$

For the original grid $P_n \square P_n$ the situation is slightly more complicated, even for the starting point $(a,b)=(0,0)$. We need to investigate the number of vertices at distance $r$ from $(0,0)$ which changes once we reach the boarder. We have
$$
|N_r(0,0)| =
\begin{cases}
4r & \text{ if } 1 \le r  \le \lceil n/2 \rceil - 1 \\
4(n-r) + O(1) & \text{ if } \lceil n/2 \rceil - 1 < r  \le 2 \lfloor n/2 \rfloor.
\end{cases}
$$
It follows from Lemma~\ref{lem:dist_r} that the number of vertices burnt is at least
\begin{eqnarray*}
\sum_{r=1}^{2 \lfloor n/2 \rfloor} \max(|N_r(0,0)| - r,0) &\ge& n^2 \left( \int_{0}^{1/2} 3x dx + \int_{1/2}^{4/5} (4-5x) dx + o(1) \right) \\
&=& n^2 \left( \frac {3}{5} + o(1) \right).
\end{eqnarray*}
We get that the fraction of vertices saved is at most $2/5 + o(1)$. Clearly, this bound can be improved. In order to play optimally and save $r$ vertices at distance $r$ during the first phase ($r \le n/2$) the firefighter has to follow the strategy described in the previous subsection. But if this is the case, the strategy is not optimal in the second phase ($r > n/2$) and there is no way to keep saving $r$ vertices at distance $r$. As we already mentioned, we conjecture that the strategy yielding the lower bound is optimal, giving the following conjecture for the case $(a,b)=(0,0)$.

\begin{conjecture}
$$
\lim_{n \to \infty} \frac {\sn(P_n \square P_n, (0,0) )}{n^2} = \frac {1}{4}.
$$
\end{conjecture}

\medskip
The proof for the general case is rather technical and we present it with all details in Appendix~\ref{sec:appenix_upper}.

\section{Infinite graphs}\label{sec:inf_graphs} 

In this section we introduce the concept of surviving rate for infinite graphs with all finite degrees, and present a few results for infinite square, hexagonal and triangular grid. Assuming a fixed and deterministic firefighter strategy, a vertex $v$ is considered \emph{saved} if the strategy guarantees that $v$ never catches fire ($v$ does not have to get protected during the process). This is well defined even for infinite graphs---given a fixed and deterministic strategy, the game is pre-determined and either there is a point of the process when $v$ catches fire or it is never on fire (that is, $v$ is saved). The surviving rate of a strategy $\mathcal{F}$ (used in the process in which the fire breaks out at vertex $v$) is then defined as $\rho_{\mathcal{F}}(G,v)=\liminf_{i\to\infty}\frac{|N_{\leq i}(v)\cap S|}{|N_{\leq i}(v)|}$ where $S=S(\mathcal{F})$ is the set of saved vertices and $N_{\leq i}(v)$ is the set of vertices at distance at most $i$ from $v$. As we assume that all degrees are finite, $|N_{\leq i}(v)|$ is always finite. Let the optimal surviving rate be $\rho(G,v)=\sup_{\mathcal{F}}\rho_{\mathcal{F}}(G,v)$. We always distinguish whether this ratio can be attained or not. Note that this coincides with the definition of $\rho(G,v)$ for finite graphs.

For example, for an infinite binary tree $T_2$ rooted at $r$, we have $\rho(T_2, r)=1$ as we can save all the vertices but a single infinite ray (path) from $r$. This follows from the fact that for trees, it is always optimal to protect a neighbor of a burning vertex rather than a vertex further away. Note that it is not possible to save all but finite number of vertices of $T_2$. Similarly, for an infinite ternary tree $T_3$ rooted at $r$ we have $\rho(T_3, r)=1/2$. We omit the proof of this statement and leave it as an exercise.

The expected surviving rate is not extensible to infinite graphs without explicitly stating the distribution (there is no uniform distribution on an infinite number of vertices). However, for vertex-transitive graphs, we have $\rho(G,v_1)=\rho(G,v_2)$ for any two vertices $v_1,v_2$, and we denote it as $\rho(G)$. Note that both square grid and hexagonal grid are vertex-transitive.

\medskip
To see the robustness of our definition, let us generalize the definition by allowing a different center of measurement:
$\rho_{\mathcal{F}}(G,v,c)=\liminf_{i\to\infty}\frac{|N_{\leq i}(c)\cap S|}{|N_{\leq i}(c)|}$ where $S$ are the vertices saved by $\mathcal{F}$, provided that the first breaks out at vertex $v$. Generally, the ratio depends on the choice of $c$
(as can be seen in $T_3$ and other fast-expanding graphs; in $T_3$ we can place $c$ to the root), but when $|N_i(c)|$ grows asymptotically strictly slower than $|N_{\leq i}(c)|$, we get the following result:

\begin{lemma}\label{lem:inf_center}
    Given an infinite connected graph $G$ with finite degrees, let $C_c(i)=|N_i(c)|$ and $A_c(i)=|N_{\leq i}(c)|$.
    If we have $C_c(i)=o(A_c(i))$\footnote{The standard notation $f(i)=o(g(i))$ denotes $f(i)/g(i) \to 0$ as $i \to \infty$} for some $c\in V(G)$, we have
    $\rho_{\mathcal{F}}(G,v,c)=\rho_{\mathcal{F}}(G,v,c')$ for any $c'\in V(G)$.
\end{lemma}

Note that this is the case for square, hexagonal, triangular and many other grid-like graphs.

\begin{proof}
    Assume fixed $c$, $c'$, $\mathcal{F}$ and $S$ and let $d=d(c',c)$.
    Then for any $i$ we have
\begin{eqnarray*}
    |N_{\leq i}(c')\cap S| &\geq& |N_{\leq i-d}(c)\cap S| \\
    &=& |N_{\leq i}(c)\cap S| - \sum_{j=i-d+1}^{i}|N_{j}(c)\cap S|  \\
    &=& |N_{\leq i}(c)\cap S| - o(|N_{\leq i}(c)|),
\end{eqnarray*}
    since $|N_{\leq i}(c)|$ is non-decreasing.
    Similarly, we get $|N_{\leq i}(c')\cap S| \leq |N_{\leq i}(c)\cap S| + o(|N_{\leq i+d}(c)|)$.
    By omitting the intersection with $S$ above, we get $|N_{\leq i}(c')| \geq |N_{\leq i}(c)| - o(|N_{\leq i}(c)|)$ and
    $|N_{\leq i}(c')| \leq |N_{\leq i}(c)| + o(|N_{\leq i+d}(c)|)$.

    Now we will show that there exists $q>0$ such that $A_c(j+1)\leq q A_c(j)$ for all $j$. For a contradiction, suppose that it is not the case. Then,
    the value of $A_c(j)$ compared to $A_c(j-1)$ at least doubles at infinitely many $j$'s.
    At these points, $C_c(j) \geq (1/2) A_c(j)$, contradicting $C_c(j)=o(A_c(j))$. Therefore we have
    $|N_{\leq i+d}(c)|\leq q^d |N_{\leq i}(c)|$ and we can replace $o(|N_{\leq i+d}(c)|)$ with
    $o(|N_{\leq i}(c)|)$ in the above expressions.

    Applying to the terms of the limit in $\rho_{\mathcal{F}}(G,v,c')$ we get
    $$
    \frac{|N_{\leq i}(c')\cap S|}{|N_{\leq i}(c')|} \leq
    \frac{|N_{\leq i}(c)\cap S|+o(|N_{\leq i}(c)|)} {|N_{\leq i}(c)|-o(|N_{\leq i}(c)|)} \leq
    \frac{|N_{\leq i}(c)\cap S|}{|N_{\leq i}(c)|} + o(1).
    $$
    Proving the other direction is analogous, and so we get $\rho_{\mathcal{F}}(G,v,c')=\rho_{\mathcal{F}}(G,v,c)$.
\end{proof}

\subsection{Infinite square grid}\label{sec:inf_square_grid} 

For the infinite square grid we show that the surviving rate is equal to $1/4$.

\begin{theorem}\label{thm:inf_sq_grid}
    For the infinite square grid $\GSq$ we have $\rho(\GSq)=1/4$.
\end{theorem}
\begin{proof}
Let $v$ be the vertex that catches fire initially. Then, after the $i$-th turn, at least $(3/4) |N_i(v)|$ vertices of $N_i(v)$ burn by Lemma~\ref{lem:dist_r}. So we have $\rho(\GSq)\leq 1/4$ by the definition of $\rho(G,v)$.

    On the other hand, the strategy outlined in Section~\ref{sec:lower_bound}, applied to the infinite grid
    saves vertices in $90^\circ$ wedge, giving $\rho(\GSq)\geq 1/4$.
\end{proof}

\subsection{Infinite hexagonal grid}\label{sec:inf_hex_grid} 

When we assume the fire starts at vertex $v_0$, the hexagonal grid is naturally divided into
six $60^\circ$ cones, see Figure~\ref{fig:hex-cones}(a).

\begin{figure}[htbp]
\begin{center}
\begin{tabular}{ccc}
	\includegraphics{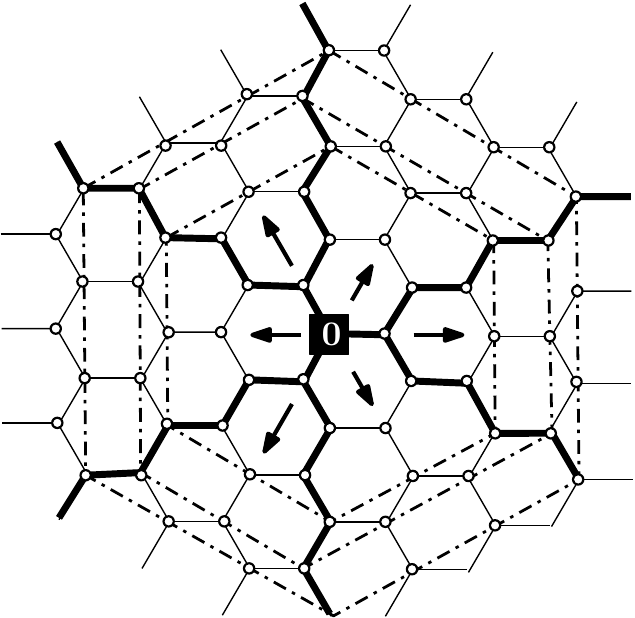} & \hspace{0cm} &
	\includegraphics{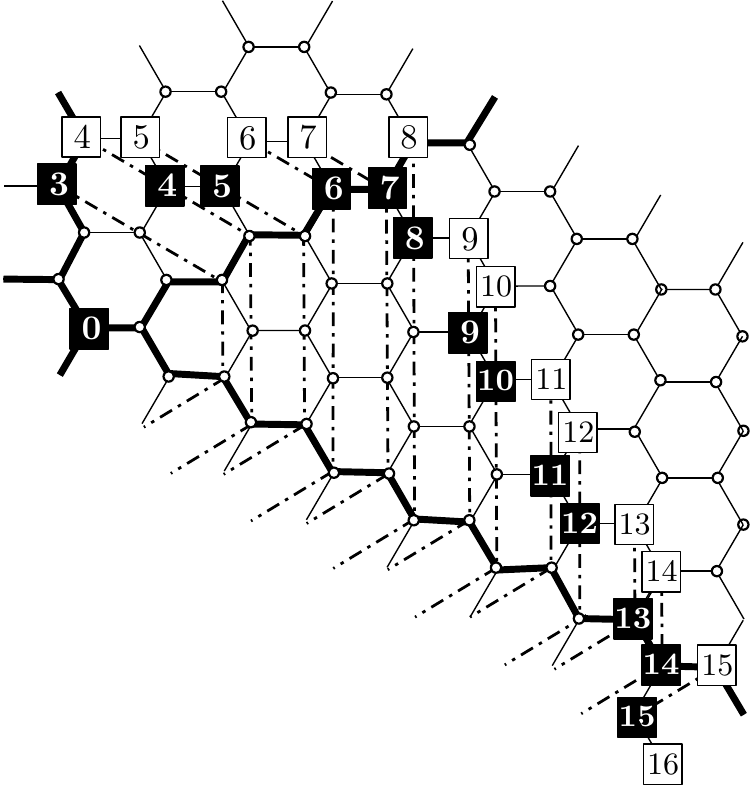} \\
      (a) && (b)
\end{tabular}
\end{center}
    \caption{(a) Six cones of the hex grid centered at $v_0$, the dash-dotted hexagons indicate $N_4(v_0)$, $N_5(v_0)$ and $N_6(v_0)$.
    (b) Construction of two segments of the spiral for $t_0=3$.}
    \label{fig:hex-cones}
\end{figure}

\smallskip
Let us start with the following simple but convenient observation that allows us to treat general situations.
At some point in the game, let $B$ be the set of burning vertices and $A\subseteq B$ be the {\em active} burning vertices,
that is vertices of $B$ with a free neighbor that can start burning in the next turn.
Let $v_0$ and $t$ be such that $A\subseteq N_{\leq t}(v_0)$. Then the firefighters may play as if exactly $N_t(v_0)$ were
burning and all the vertices of $B\setminus A$ were protected (that is the new setting) and any strategy playable in the new setting
will be playable in the original situation as well.

This follows from the fact that only the vertices of $A$ will ever cause new fire
and assuming more vertices to be burning only helps the fire. Interchanging burning non-active fires with protected vertices
makes no difference for the progress of the game (except for counting saved vertices), as non-active vertices will never spread fire.
Below, this allows us to assume that the already burnt area is some $N_{\leq t}(v_0)$.

\medskip
{\bf Spiral construction.} 
Assuming that $N_{\leq t_0}(v_0)$ burn, we show how one firefighter can build a spiral of protected vertices delaying the fire by a
constant factor. Without loss of generality, let $t_0+1$ be the number of the next turn to take place. This matches the situation
of fire starting at $v_0$ in turn 0.

The spiral is composed of successive segments, each segment a line contained in one of the six cones.
In turn $t$, the firefighter will protect a vertex in $N_t(v_0)$, so it can not be on fire at that time.
The construction of two successive segments is illustrated on Figure~\ref{fig:hex-cones}(b) for the case $t_0=3$.
Starting a segment in turn $t$ on one cone boundary means that the next segment (starting at another cone boundary)
will start in turn $2t$ or $2t+1$ depending on the cone type, however, $2t+O(1)$ is sufficient for our purposes.

\begin{figure}[hbt]
\begin{center}
    \begin{tabular}{p{6cm}p{0.1cm}p{6cm}}
    \strut
    \vskip 1cm
    \centerline{\includegraphics{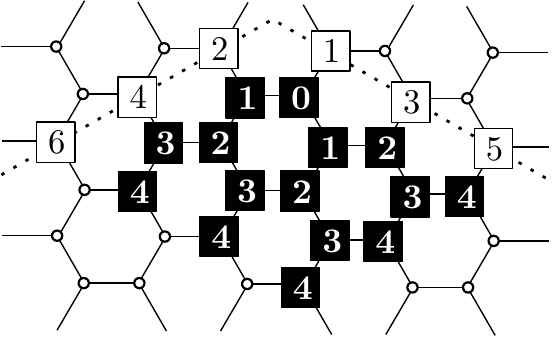}}
    \centerline{(a)}
    \vskip 1cm
    \centerline{\includegraphics{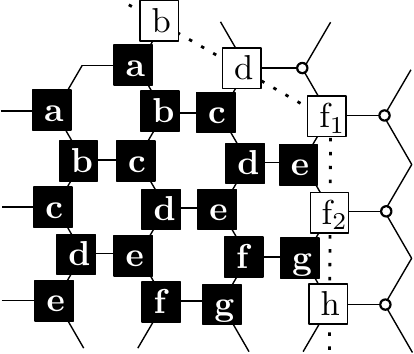}}
    \centerline{(b)}
    \strut
    &&
    \strut\vskip 1cm
    \centerline{\includegraphics{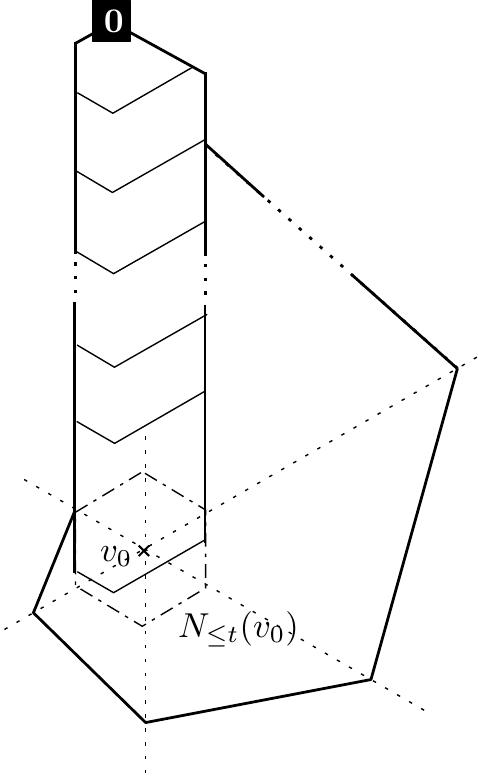}}
    \centerline{(c)}
    \strut
    \end{tabular}
\end{center}
\caption{(a) Start of the strategy protecting $2/3$ of the grid with indicated protected rays.
    (b) Bending the protected ray with one extra firefighter at time $f$.
    (c) Overview of the strategy of Theorem~\ref{thm:hex}. Starting with two bends, we build
    a long enough strip, then start spiraling. (Note that the spiral is deformed to fit in the figure.)
    $v_0$ and $N_{\leq t}(v_0)$ are as in the spiral construction description in Section~\ref{sec:inf_hex_grid},
    the thin lines indicate the active fire at certain time points.}
    \label{fig:hex}
\end{figure}

\medskip
\begin{proof}[Proof of Theorem~\ref{thm:hex}]
    The initial strategy is best illustrated and explained by Figure~\ref{fig:hex}(a).
    The firefighter alternates between protecting two rays, always playing to $N_t(v_0)$ in turn $t$,
    therefore making the strategy valid. Note the rays are chosen such that one intersects only
    $N_t(v_0)$ with $t$ odd and the other with $t$ even.

    At the point we get the first extra protected vertex, we can {\em bend} one of the rays by $60^\circ$ as indicated in
    Figure~\ref{fig:hex}(b). The letters indicate consecutive turn numbers, in the turns $a$, $c$, $e$ and $g$,
    the firefighter protects the other ray, $f$ is the turn at time $t_1$ with the first extra firefighter.
    Note that from this point on, 5/6 of the grid is protected --- in the limit, the fire would only occupy
    a $60^\circ$ wedge.

    When we get the second extra firefighter at time $t_2$, we bend the other ray in a symmetric way.
    Note that in case the extra firefighter came in a turn we play on the other ray, we can protect the desired vertex
    one turn in advance.
   
    After the two extra protections, the fire is restricted to a strip extending only in one direction as in Figure~\ref{fig:hex}(c).
    Note that the strip might be very wide, depending on $t_1$ and $t_2$, but the width does not change and
    there is some $t$ such that in every turn, all the active fire is contained in some ball of radius $t$.
    We let the strip grow to length at least $2^7 t$ to make the following step possible.

    The firefighter stops protecting the two rays and start building a spiral around $N_t(v_0)$ for suitable $v_0$ as
    described in Section~\ref{sec:inf_hex_grid}.
    If we start the spiral in the angle indicated in Figure~\ref{fig:hex}(c), the fifth spiral segment
    hits the wall of the strip and we have enclosed the fire with protected vertices. Since the first
    segment of the spiral starts in distance $t$ from $v_0$, the fifth segment ends in distance
    $2^6t+O(1)$ from $v_0$, so stripe length $2^7t$ is enough for the spiral and the strip boundary to meet.
\end{proof}

\begin{figure}[htbp]
    \begin{center}
    \begin{tabular}{m{9.5cm}m{0.1cm}m{5.5cm}}
        \centerline{\includegraphics{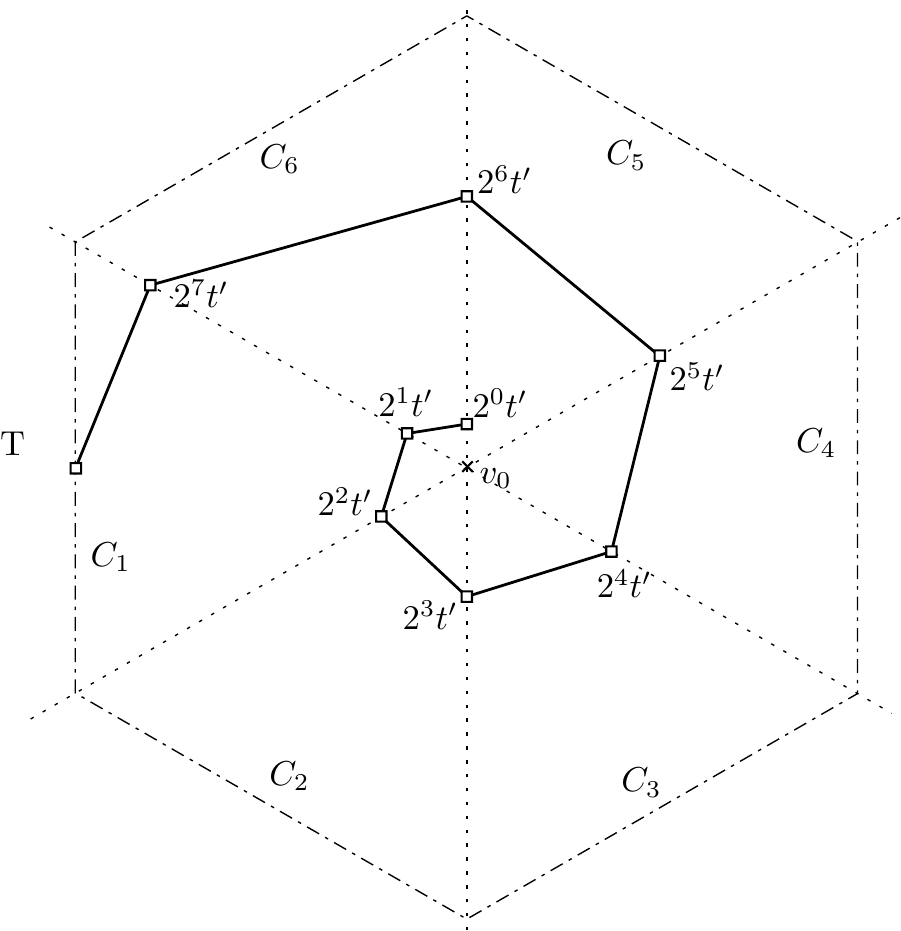}} &&
        \centerline{\includegraphics{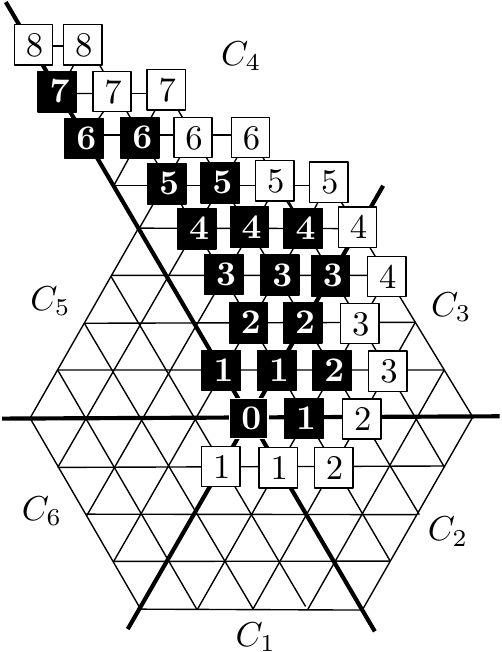}}
	\\
	\centerline{(a)} &&
	\centerline{(b)}
    \end{tabular}
    \end{center}
    \caption{(a) Seven consecutive segments of a spiral. The numbers indicate the distances of the segment ends
    from $v_0$ up to $O(1)$ additive factor. Note that the spiral is deformed to fit the figure.
    (b) Cones $C_1,\dots C_6$ in $\GTri$ and the construction of first four segments of a spiral in case $t=0$.}
    \label{fig:hex-spiral}
    \label{fig:triangle-spiral}
\end{figure}

\begin{proof}[Proof of Theorem~\ref{thm:hex_slowdown}]
    Choose $v_0$ and $t$ so that all burning vertices are contained in $N_{\le t}(v_0)$ and
    start building a spiral in distance $t+1$ from $v_0$ as described above. We may, without loss of generality,
    assume that the construction starts in turn $t+1$.

    Now consider turn $T$ and assume $T\geq 2^8t$.
    Let $t'$ be such that $t'=2^kt+O(k)$ for some $k$ be a length of a segment of the spiral
    and $2^7t'+O(k) \leq T < 2^8t'+O(k)$. The differences $O(k)$ come from the fact that
    the spiral segment lengths do not double exactly but up to $O(1)$.
    The situation with seven largest segments is illustrated on Figure~\ref{fig:hex-spiral}~(a).

    Notice that even when considering only the last seven consecutive segments of the spiral as protected,
    the fire has to take a detour proportional to the size of the smallest considered segment, which is $t'$,
    to reach any vertex of $N_{T}(v_0)$.

    More precisely, let $C_1, C_2, \dots,C_6$ denote the cones around $v_0$ as on Figure~\ref{fig:hex-spiral}~(a).
    Let $B_i=N_{T}(v_0)\cap C_i$.
    To reach $B_i$ from $v_0$, the fire has to travel at least the distance $v_0 \cdots 2^{i-1}t' \cdots B_i$.
    Any such path has length at least $2^{i-1}t'+O(1)+T$ which is minimal for $B_1$.
    The actual grid distance $v_0 \dots B_i$ is always $T$.

    Therefore, we can take any constant $c$ such that for $T$ large enough we have $c>\frac{T}{2^{0}t'+O(1)+T}$.
    Since $t'>2^{-8}T$, this is satisfied by any $c>\frac{1}{1+2^{-8}}\approx 0.9961$ for $T$ large enough compared to $t$.
\end{proof}

\subsection{Infinite triangular grid}\label{sec:inf_tringle_grid} 

In this section we prove results similar to those on hexagonal grid. The triangular grid shares many characteristics with the hexagonal one,
Figure~\ref{fig:triangle-spiral}~(b) shows a definition of a cone, and the constructions are generally very similar in nature, allowing
us to reuse certain arguments from Section~\ref{sec:inf_hex_grid}, which we do for the sake of brevity.

\begin{proof}[Proof of Theorem~\ref{thm:triangle_slowdown}]
    The proof follows the structure of the proof of Theorem~\ref{thm:hex_slowdown}, please refer to it for some of the common details.

    We assume that all fire is contained within $N_{\leq t}(v_0)$ for some $t$ and $v_0$ and assume that the next turn has number $t+1$.
    We then start building a spiral as illustrated on Figure~\ref{fig:triangle-spiral}~(b) for $t=0$ (but the general case is
    analogous to the construction in $C_2$ and $C_3$). Note that we also have that if we start the construction on the boundary of $C_i$ in
    turn $t$, we reach the other boundary of $C_i$ in turn $2t$. Again, the correctness of the construction follows from the fact
    that in turn $k$, the firefighters play to vertices in $N_{k}(v_0)$, avoiding any possibly burning vertices.

    The rest of the proof is almost identical to that of Theorem~\ref{thm:hex_slowdown}, namely the choice of $t'$, the overall situation
    within the cones is identical to Figure~\ref{fig:hex-spiral}~(a), and the resulting constant can be chosen within the same interval,
    that is $c>\frac{1}{1+2^{-8}}\approx 0.9961$ for $T$ large enough. 
\end{proof}


\section{Acknowledgement}

The authors would like to thank the anonymous referee for pointing out the relationship between strategies for $\GHex$ and for $\GTri$. While this relation is unpublished, it is quite similar to the relation between the square lattice and the square lattice with diagonals, which is described in~\cite{32fraction} (see Proposition~11). 

\smallskip

\noindent \textbf{Claim.} \emph{Every strategy given $2x_i$ firefighters on the $i$-th turn ($i = 1, 2, \ldots, t$) on $\GTri$ which stops the fire from spreading after $t$ turns leaving $b$ burnt vertices, can be translated into a strategy to stop the fire in $2t + 1$ turns on $\GHex$ using $x_{\lceil i/2 \rceil}$ firefighters on the $i$-th turn ($i=1,2,\ldots, 2t$) and no  firefighters on the last turn, leaving at most $2b+f$ burning vertices ($f$ is the number of firefighters used in total, which is the same for both strategies).}

\begin{proof}
Let 
$$
A = \left\{ a i +b \left( \frac {\sqrt{3}}{2} + \frac i2 \right) + c \left( \frac {\sqrt{3}}{2} -\frac i2 \right) : a, b, c \in \Z \right)
$$
be the collection of the vertices of the triangular lattice $\GTri$, represented by complex numbers. Here, two vertices share an edge if and only if they are at distance 1. Let
$$
B = A \cup \left\{a + \frac {1}{\sqrt{3}} : a \in A \right\},
$$
and observe the graph whose vertices are the elements of $B$ and two vertices are connected by an edge if and only if they are at distance $1/\sqrt{3}$ is exactly the hexagonal lattice $\GHex$. (See Figure~\ref{fig:fig-hex-triangles}.)

\begin{figure}[htbp]
\begin{center}
      \includegraphics[height=2.1in]{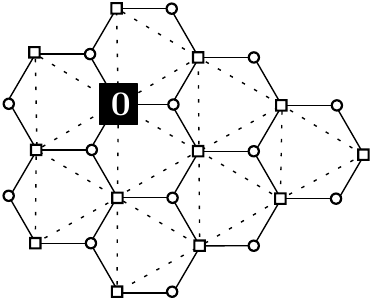}
\end{center}
    \caption{Relationship between strategies for $\GHex$ and for $\GTri$.} \label{fig:fig-hex-triangles}
\end{figure}

Without loss of generality, assume that the fire starts at 0, and we are given a winning strategy on $\GTri$ (that is, a strategy on $A$) which uses $2x_i$ firefighters on the $i$-th turn and contains the fire after $t$ turns. We will mimic this strategy to get another one for $\GHex$ (that is, a strategy on $B \supseteq A$) that protects vertices of $A$ only. We observe that for any such strategy the following properties hold:
\begin{itemize}
\item in every odd turn, the fire spreads to some vertices of $B \setminus A$ that are unprotected and at distance $1/\sqrt{3}$ from some vertices of $A$ that are burning; no vertex of $A$ catches fire,
\item in every even turn, the fire spreads to some vertices of $A$ that are unprotected and at distance 1 from some vertices of $A$ that are burning (that is, as if we were playing on $A$); no vertex of $B$ catches fire.
\end{itemize}
We play in the following way on $\GHex$. On the turns $2i-1$ and $2i$ we place $x_i$ firefighters on the vertices indicated by the strategy for $\GTri$; the order in which we place them (that is, which vertices should be protected on each of the two turns) is not important. Observing that the fire spreads to vertices of $A$ only every second turn, and that while playing using this strategy it spreads exactly as though it would have on $A$, we get that after $t$ rounds consisting of two turns, the fire will never spread again to a vertex in $A$ and thus after one additional turn it will stop spreading completely. 

We know that $b$ vertices of $A$ are burnt. It remains to calculate how many vertices are burnt in total. Here we use the following observation. We know that every burning vertex $a \in B \setminus A$ has the property that the three vertices of $A$ at distance $1/\sqrt{3}$ from $a$ are either burning or protected. One of them, say, $a - 1/\sqrt{3}$ is unique for $a$ so the number of vertices of $B \setminus A$ that are burnt is at most $b$. The proof is finished.
\end{proof}


\appendix 
\section{Proof of the upper bound from Theorem~\ref{thm:finite}}\label{sec:appenix_upper}

Here we examine the upper bound in the general case when a fire breaks out at a vertex $(a,b)$ of $P_n \square P_n$. As before, due to the symmetry, we may assume that $a=xn+o(1)$, $b=yn+o(1)$ for some $0 \le x \le 1/2$ and $0 \le y \le x$. We want to investigate the size of $N_r(a,b)$ so we are interested in two things: the time $(t_i + o(1)) n$ when we reach each of the 4 borders ($i = N, S, E, W$), and the time $(t_i+o(1)) n$ when each of the 4 fire fronts ($i = NE, SE, SW, NW$) disappear. Clearly,
$$
t_E = \frac 12 - x, \ \ t_N = \frac 12 - y, \ \ t_S = \frac 12 + y, \ \ t_W = \frac 12 + x,
$$
and $t_E \le t_N \le t_S \le t_W$. Moreover,
$$
t_{NE} = 1-x-y, \ \ t_{SE} = 1-x+y, \ \ t_{NW} = 1+x-y, \ \ t_{SW} = 1+x+y,
$$
and $t_{NE} \le t_{SE} \le t_{NW} \le t_{SW}$. The order in which these events occur determines the formula for $|N_r(a,b)|$. It is easy to see that $t_{NE} \ge t_N$, $t_{SE} \ge t_S$, and $t_{NW} \ge t_W$, and so there are only 5 cases to consider. Unfortunately, since the number of vertices burnt is at least $\sum_{r \ge 1} \max(|N_r(0,0)| - r,0)$ (by Lemma~\ref{lem:dist_r}), sometimes we need to consider some sub-cases depending on in which time interval $|N_r(0,0)| - r$ becomes negative. Let $(t+o(1))n$ be the first time this happens. The calculations are elementary but quite tedious, so we refer the reader to the Maple worksheet~\cite{maple} to check integrals, etc.

\begin{figure}[htbp]
    \centerline{\includegraphics[scale=0.9]{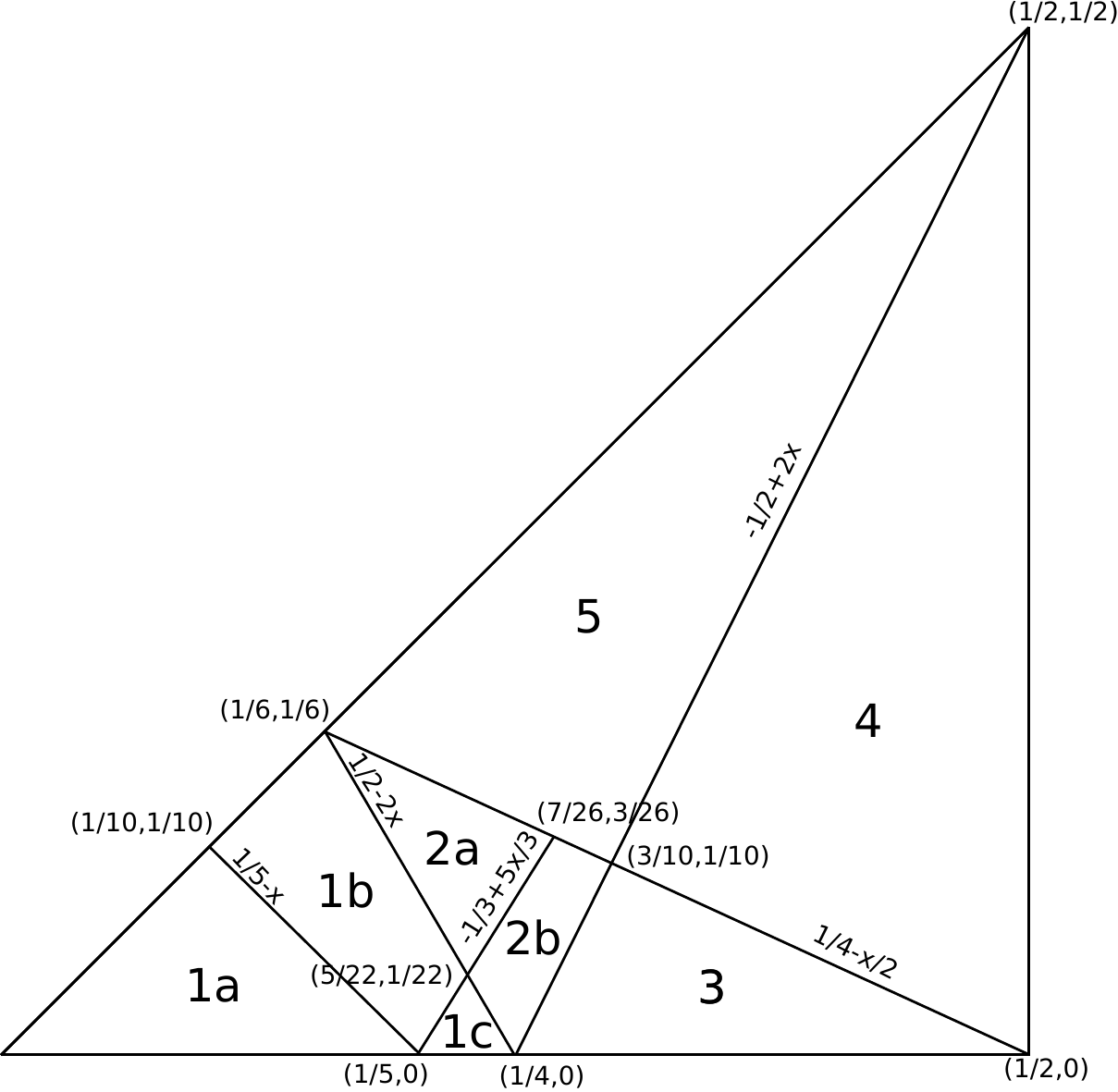}}
    \caption{The regions corresponding to the individual cases.}
    \label{fig:regions}
\end{figure}

\bigskip

\textbf{Case 1: $t_E \le t_N \le t_S \le t_W \le t_{NE} \le t_{SE} \le t_{NW} \le t_{SW}$.}\\
The condition $t_W \le t_{NE}$ is equivalent to $y \le 1/2 - 2x$ so we are concerned with the Region~1 presented on Figure~\ref{fig:regions}. The number of vertices at distance $rn$ from $(a,b)$ (for some $r=r(n)$) behaves as follows:
{ \tiny
$$
\frac {|N_{rn}(a,b)|}{n} = o(1) +
\begin{cases}
4r & \text{ if } 0 \le r  \le t_E \\
1-2x+2r & \text{ if } t_E \le r \le t_N\\
2-2x-2y & \text{ if } t_N \le r \le t_S\\
3-2x-2r & \text{ if } t_S \le r \le t_W\\
4-4r & \text{ if } t_W \le r \le t_{NE}\\
3+x+y-3r & \text{ if } t_{NE} \le r \le t_{SE}\\
2+2x-2r & \text{ if } t_{SE} \le r \le t_{NW}\\
1+x+y-r & \text{ if } t_{NW} \le r \le t_{SW}.
\end{cases}
$$
}

\textbf{Case 1a: $t_W \le t \le t_{NE}$.}\\
The condition $t \le t_{NE}$ implies that $y \le 1/5-x$ (as before, we direct the reader to Figure~\ref{fig:regions}). Provided $(a,b)$ is from the region we consider in this sub-case, $t = 4/5$ (note that $|N_{rn}(a,b)| - rn = (4-5r + o(1))n$ for $t_W \le r \le t_{NE}$). In other words, at time $(4/5+o(1))n$, Lemma~\ref{lem:dist_r} stops working in the sense that it does not give us any non-trivial bound for the number of vertices burning. We stop investigating the process at that time and assume the worst case scenario that all remaining vertices will become protected. Hence, the proportion of vertices burnt is at least
\begin{eqnarray*}
B_{1a} (x,y) &:=& \int_0^{t_E} 3r dr + \int_{t_E}^{t_N} (1-2x+r)dr + \int_{t_N}^{t_S} (2-2x-2y-r) dr  \\
& & + \int_{t_S}^{t_W} (3-2x-3r) dr + \int_{t_W}^t (4-5r)dr = \frac 35 - 2x^2 - 2y^2.
\end{eqnarray*}
(Indeed, in the first time interval, there are $4rn$ vertices at distance $rn$ from $(a,b)$ and at most $rn$ of them are protected by Lemma~\ref{lem:dist_r}. In the second time interval, there are $(1-2x+2r+o(1))n$ vertices at distance $rn$ from $(a,b)$ but the lemma still gives us that at most $rn$ of them are protected. Hence the term $1-2x+r$. Other terms are derived in a similar way.) The contribution from vertices from this region is calculated as follows (as usual, we refer to Figure~\ref{fig:regions}).
$$
C_{1a} := \int_0^{1/10} \int_0^x B_{1a} (x,y) dy dx + \int_{1/10}^{1/5} \int_0^{1/5-x} B_{1a} (x,y) dy dx = \frac {43}{7500} \approx 0.005733.
$$
(The region has a triangular shape so we have to split it into two integrals.)

\textbf{Case 1b: $t_{NE} \le t \le t_{SE}$.}\\
The condition $t \le t_{SE}$ implies that $y \ge -1/3+5x/3$. It follows that $t=3/4+x/4+y/4$ and the proportion of vertices burnt is at least
\begin{eqnarray*}
B_{1b} (x,y) &:=& \int_0^{t_E} 3r dr + \int_{t_E}^{t_N} (1-2x+r)dr + \int_{t_N}^{t_S} (2-2x-2y-r) dr  \\
& & + \int_{t_S}^{t_W} (3-2x-3r) dr + \int_{t_W}^{t_{NE}} (4-5r)dr + \int_{t_{NE}}^t (3+x+y-4r) dr \\
&=& \frac 58 - \frac{x}{4} - \frac {11x^2}{8} - \frac {y}{4} - \frac {11y^2}{8} + \frac {5xy}{4}.
\end{eqnarray*}
The contribution from vertices from this region is
\begin{eqnarray*}
C_{1b} &:=& \int_{1/10}^{1/6} \int_{1/5-x}^x B_{1b} (x,y) dy dx + \int_{1/6}^{1/5} \int_{1/5-x}^{1/2-2x} B_{1b} (x,y) dy dx \\
&& + \int_{1/5}^{5/22} \int_{-1/3+5x/3}^{1/2-2x} B_{1b} (x,y) dy dx = \frac {459563}{89842500} \approx 0.005115.
\end{eqnarray*}

\textbf{Case 1c: $t_{SE} \le t \le t_{NW}$.}\\
The condition $t \le t_{NW}$ is equivalent to $y \le 1/3+x/3$ which is satisfied by all points in the Case 1. It follows that $t=2/3 + 2x/3$ and the proportion of vertices burnt is at least
\begin{eqnarray*}
B_{1c} (x,y) &:=& \int_0^{t_E} 3r dr + \int_{t_E}^{t_N} (1-2x+r)dr + \int_{t_N}^{t_S} (2-2x-2y-r) dr  \\
& & + \int_{t_S}^{t_W} (3-2x-3r) dr + \int_{t_W}^{t_{NE}} (4-5r)dr + \int_{t_{NE}}^{t_{SE}} (3+x+y-4r) dr \\
& & + \int_{t_{SE}}^t (2+2x-3r) dr = \frac 23 - \frac {2x}{3} - \frac {x^2}{3} - y^2.
\end{eqnarray*}
The contribution from vertices from this region is
\begin{eqnarray*}
C_{1c} &:=& \int_{1/5}^{5/22} \int_0^{-1/3+5x/3} B_{1c} (x,y) dy dx + \int_{5/22}^{1/4} \int_0^{1/2-2x} B_{1c} (x,y) dy dx \\
&=& \frac {434549}{766656000} \approx 0.000567.
\end{eqnarray*}

\textbf{Case 2: $t_E \le t_N \le t_S \le t_{NE} \le t_W \le t_{SE} \le t_{NW} \le t_{SW}$.}\\
The condition $t_S \le t_{NE}$ is equivalent to $y \le 1/4 - x/2$ and the condition $t_W \le t_{SE}$ to $y \ge -1/2 + 2x$, so we are concerned with the Region~2 presented on Figure~\ref{fig:regions}. This time
{ \tiny
$$
\frac {|N_{rn}(a,b)|}{n} = o(1) +
\begin{cases}
4r & \text{ if } 0 \le r  \le t_E \\
1-2x+2r & \text{ if } t_E \le r \le t_N\\
2-2x-2y & \text{ if } t_N \le r \le t_S\\
3-2x-2r & \text{ if } t_S \le r \le t_{NE}\\
2-x+y-r & \text{ if } t_{NE} \le r \le t_{W}\\
3+x+y-3r & \text{ if } t_{W} \le r \le t_{SE}\\
2+2x-2r & \text{ if } t_{SE} \le r \le t_{NW}\\
1+x+y-r & \text{ if } t_{NW} \le r \le t_{SW}.
\end{cases}
$$
}

\textbf{Case 2a: $t_{W} \le t \le t_{SE}$.}\\
The condition $t \le t_{SE}$ implies that $y \ge -1/3+5x/3$. It follows that $t=3/4+x/4+y/4$ and the proportion of vertices burnt is at least
\begin{eqnarray*}
B_{2a} (x,y) &:=& \int_0^{t_E} 3r dr + \int_{t_E}^{t_N} (1-2x+r)dr + \int_{t_N}^{t_S} (2-2x-2y-r) dr  \\
& & + \int_{t_S}^{t_{NE}} (3-2x-3r) dr + \int_{t_{NE}}^{t_W} (2-x+y-2r) dr \\
&& + \int_{t_W}^{t} (3+x+y-4r) dr \\
&=& \frac 58 - \frac{x}{4} - \frac {11x^2}{8} - \frac {y}{4} - \frac {11y^2}{8} + \frac {5xy}{4} = B_{1b} (x,y).
\end{eqnarray*}
The contribution from vertices from this region is
\begin{eqnarray*}
C_{2a} &:=& \int_{1/6}^{5/22} \int_{1/2-2x}^{1/4-x/2} B_{2a} (x,y) dy dx + \int_{5/22}^{7/26} \int_{-1/3+5x/3}^{1/4-x/2} B_{2a} (x,y) dy dx \\
&=& \frac {358687}{157907178} \approx 0.002275.
\end{eqnarray*}

\textbf{Case 2b: $t_{SE} \le t \le t_{NW}$.}\\
The condition $t \le t_{NW}$ is equivalent to $y \le 1/3+x/3$ which is satisfied by all points in the Case 2. It follows that $t=2/3 + 2x/3$ and the proportion of vertices burnt is at least
\begin{eqnarray*}
B_{2b} (x,y) &:=& \int_0^{t_E} 3r dr + \int_{t_E}^{t_N} (1-2x+r)dr + \int_{t_N}^{t_S} (2-2x-2y-r) dr  \\
& & + \int_{t_S}^{t_{NE}} (3-2x-3r) dr + \int_{t_{NE}}^{t_W} (2-x+y-2r) dr \\
&& + \int_{t_W}^{t_{SE}} (3+x+y-4r) dr + \int_{t_{SE}}^t (2+2x-3t) dr \\
&=& \frac 23 - \frac {2x}{3} - \frac {x^2}{3} - y^2 = B_{1c} (x,y).
\end{eqnarray*}
The contribution from vertices from this region is
\begin{eqnarray*}
C_{2b} &:=& \int_{5/22}^{1/4} \int_{1/2-2x}^{-1/3+5x/3} B_{2b} (x,y) dy dx + \int_{1/4}^{7/26} \int_{-1/2+2x}^{-1/3+5x/3} B_{2b} (x,y) dy dx \\
&& + \int_{7/26}^{3/10} \int_{-1/2+2x}^{1/4-x/2} B_{2b} (x,y) dy dx  = \frac {478988221}{280723872000} \approx 0.001706.
\end{eqnarray*}

\textbf{Case 3: $t_E \le t_N \le t_S \le t_{NE} \le t_{SE} \le t_W \le t_{NW} \le t_{SW}$.}\\
In this case, we are concerned with the Region~3 presented on Figure~\ref{fig:regions}. This time
{ \tiny
$$
\frac {|N_{rn}(a,b)|}{n} = o(1) +
\begin{cases}
4r & \text{ if } 0 \le r  \le t_E \\
1-2x+2r & \text{ if } t_E \le r \le t_N\\
2-2x-2y & \text{ if } t_N \le r \le t_S\\
3-2x-2r & \text{ if } t_S \le r \le t_{NE}\\
2-x+y-r & \text{ if } t_{NE} \le r \le t_{SE}\\
1 & \text{ if } t_{SE} \le r \le t_{W}\\
2+2x-2r & \text{ if } t_{W} \le r \le t_{NW}\\
1+x+y-r & \text{ if } t_{NW} \le r \le t_{SW}.
\end{cases}
$$
}
In this case, $t_W \le t \le t_{NW}$. It follows that $t=2/3 + 2x/3$ and the proportion of vertices burnt is at least
\begin{eqnarray*}
B_{3} (x,y) &:=& \int_0^{t_E} 3r dr + \int_{t_E}^{t_N} (1-2x+r)dr + \int_{t_N}^{t_S} (2-2x-2y-r) dr  \\
& & + \int_{t_S}^{t_{NE}} (3-2x-3r) dr + \int_{t_{NE}}^{t_{SE}} (2-x+y-2r) dr \\
&& + \int_{t_{SE}}^{t_{W}} (1-r) dr + \int_{t_{W}}^t (2+2x-3t) dr \\
&=& \frac 23 - \frac {2x}{3} - \frac {x^2}{3} - y^2 = B_{2b}(x,y) = B_{1c} (x,y).
\end{eqnarray*}
The contribution from vertices from this region is
$$
C_{3} := \int_{1/4}^{3/10} \int_0^{-1/2+2x} B_{3} (x,y) dy dx + \int_{3/10}^{1/2} \int_0^{1/4-x/2} B_{3} (x,y) dy dx  = \frac {2807}{576000} \approx 0.004873.
$$

\textbf{Case 4: $t_E \le t_N \le t_{NE} \le t_S \le t_{SE} \le t_W \le t_{NW} \le t_{SW}$.}\\
In this case, we are concerned with the Region~4 presented on Figure~\ref{fig:regions}. This time
{ \tiny
$$
\frac {|N_{rn}(a,b)|}{n} = o(1) +
\begin{cases}
4r & \text{ if } 0 \le r  \le t_E \\
1-2x+2r & \text{ if } t_E \le r \le t_N\\
2-2x-2y & \text{ if } t_N \le r \le t_{NE}\\
1-x-y+r & \text{ if } t_{NE} \le r \le t_{S}\\
2-x+y-r & \text{ if } t_{S} \le r \le t_{SE}\\
1 & \text{ if } t_{SE} \le r \le t_{W}\\
2+2x-2r & \text{ if } t_{W} \le r \le t_{NW}\\
1+x+y-r & \text{ if } t_{NW} \le r \le t_{SW}.
\end{cases}
$$
}
In this case, $t_W \le t \le t_{NW}$. It follows that $t=2/3 + 2x/3$ and the proportion of vertices burnt is at least
\begin{eqnarray*}
B_{4} (x,y) &:=& \int_0^{t_E} 3r dr + \int_{t_E}^{t_N} (1-2x+r)dr + \int_{t_N}^{t_{NE}} (2-2x-2y-r) dr  \\
& & + \int_{t_{NE}}^{t_{S}} (1-x-y) dr + \int_{t_{S}}^{t_{SE}} (2-x+y-2r) dr \\
&& + \int_{t_{SE}}^{t_{W}} (1-r) dr + \int_{t_{W}}^t (2+2x-3t) dr \\
&=& \frac 23 - \frac {2x}{3} - \frac {x^2}{3} - y^2 = B_3(x,y) = B_{2b}(x,y) = B_{1c} (x,y).
\end{eqnarray*}
The contribution from vertices from this region is
$$
C_{4} := \int_{3/10}^{1/2} \int_{1/4-x/2}^{-1/2+2x} B_{4} (x,y) dy dx = \frac {473}{36000} \approx 0.013139.
$$

\textbf{Case 5: $t_E \le t_N \le t_{NE} \le t_{S} \le t_W \le t_{SE} \le t_{NW} \le t_{SW}$.}\\
In this case, we are concerned with the Region~5 presented on Figure~\ref{fig:regions}. This time
{ \tiny
$$
\frac {|N_{rn}(a,b)|}{n} = o(1) +
\begin{cases}
4r & \text{ if } 0 \le r  \le t_E \\
1-2x+2r & \text{ if } t_E \le r \le t_N\\
2-2x-2y & \text{ if } t_N \le r \le t_{NE}\\
1-x-y+r & \text{ if } t_{NE} \le r \le t_{S}\\
2-x+y-r & \text{ if } t_{S} \le r \le t_{W}\\
3+x+y-3r & \text{ if } t_{W} \le r \le t_{SE}\\
2+2x-2r & \text{ if } t_{SE} \le r \le t_{NW}\\
1+x+y-r & \text{ if } t_{NW} \le r \le t_{SW}.
\end{cases}
$$
}
In this case, $t_S \le t \le t_{W}$. It follows that $t=1-x/2+y/2$ and the proportion of vertices burnt is at least
\begin{eqnarray*}
B_{5} (x,y) &:=& \int_0^{t_E} 3r dr + \int_{t_E}^{t_N} (1-2x+r)dr + \int_{t_N}^{t_{NE}} (2-2x-2y-r) dr  \\
& & + \int_{t_{NE}}^{t_{S}} (1-x-y) dr + \int_{t_{S}}^{t} (2-x+y-2r) dr \\
&=& \frac 34 - x - \frac {x^2}{4} - \frac{5y^2}{4} + \frac {xy}{2}.
\end{eqnarray*}
The contribution from vertices from this region is
$$
C_{5} := \int_{1/6}^{3/10} \int_{1/4-x/2}^{x} B_{5} (x,y) dy dx + \int_{3/10}^{1/2} \int_{-1/2+2x}^{x} B_{5} (x,y) dy dx= \frac {1907}{162000} \approx 0.011772.
$$

Finally, since there are 8 symmetric regions to consider,
\begin{eqnarray*}
\rho(P_n \square P_n) &\le& 1-8(C_{1a}+C_{1b}+C_{1c}+C_{2a}+C_{2b}+C_3+C_4+C_5) + o(1) \\
&=& \frac {67243}{105300} + o(1) < 0.6386.
\end{eqnarray*}
The proof of the upper bound is finished.

\end{document}